\documentclass{article}

\usepackage{amsmath}
\usepackage{amsfonts,amssymb}
\usepackage{amsthm}
\usepackage{amscd}

\usepackage{enumitem}
\usepackage{graphicx}%

\usepackage{nth}

\theoremstyle{plain} \numberwithin{equation}{section}
\newtheorem{theorem}{Theorem}[section]

\newtheorem{lemma}[theorem]{Lemma}
\newtheorem{proposition}[theorem]{Proposition}
\theoremstyle{definition}

\newtheorem{remark}[theorem]{Remark}

\begin{document}
	\title{Spacing properties of the zeros of orthogonal polynomials on Cantor sets via a sequence of polynomial mappings}
	\author{G\"{o}kalp Alpan\footnote{The author is supported by a grant from T\"{u}bitak: 115F199.}\\
	Department of Mathematics\\
		Bilkent University\\
		06800, \c{C}ankaya\\
		Ankara, Turkey\\ e-mail: \texttt{gokalp@fen.bilkent.edu.tr}\\}
	
\date{}

	\maketitle
\begin{abstract}
Let $\mu$ be a probability measure with an infinite compact support on $\mathbb{R}$. Let us further assume that $(F_n)_{n=1}^\infty$ is a sequence of orthogonal polynomials for $\mu$ where $(f_n)_{n=1}^\infty$  is a sequence of nonlinear polynomials and $F_n:=f_n\circ\dots\circ f_1$ for all $n\in\mathbb{N}$. We prove that if there is an $s_0\in\mathbb{N}$ such that $0$ is a root of $f_n^\prime$ for each $n>s_0$ then the distance between any two zeros of an orthogonal polynomial for $\mu$ of a given degree greater than $1$ has a lower bound in terms of the distance between the set of critical points and the set of zeros of some $F_k$. Using this, we find sharp bounds from below and above for the infimum of distances between the consecutive zeros of orthogonal polynomials  for singular continuous measures. 
\end{abstract}


\section{Introduction}
In the last ten years, there has been an explosion of interest in spacing of the zeros of orthogonal polynomials on the real line. For probability (unit Borel) measures having a non-trivial absolutely continuous part (with respect to the Lebesgue measure on $\mathbb{R}$), there are many results (see e.g. \cite{avila,last,levlub,sim1,totiksze,varga}) concerning the fine structure of the zeros of orthogonal polynomials. Simon-Kruger, in \cite{kruger}, discuss the zero spacing of the orthogonal polynomials for the Cantor-Lebesgue measure of the Cantor ternary set. To our knowledge, there was no prior work except \cite{kruger} on the structure of zeros for the purely singular continuous measure case. Here, our main aim is to give some examples of singular continuous measures for which the minimal distance between the consecutive zeros of the associated orthogonal polynomials can be approximated accurately. 

Throughout, measures that we consider are probability measures, unless specified otherwise, with a compact support in $\mathbb{C}$ and we set $\mathbb{N}=\{1,2,\dots\}$ and $\mathbb{N}_0=\mathbb{N}\cup \{0\}$. Let $\mu$ be a measure with $\mathrm{Card(supp(\mu))}\geq n>1$ for some $n\in\mathbb{N}$. Then for each $m\in\mathbb{N}$ with $m\leq n-1$, the monic polynomial $P_m(\cdot; \mu)$ of degree $m$ satisfying $$\displaystyle \|P_m(\cdot;\mu)\|_{L_2(\mathbb{C},\mu)}^2= \inf_{Q_m \in \mathcal{P}_m}\int |Q_m(z)|^2\,d\,\mu(z)>0,$$ is called the $m$-th monic orthogonal polynomial for $\mu$.
Here, $\| \cdot \|_{L_2(\mathbb{C},\mu)}$ is the standard norm in $L_2(\mathbb{C},\mu)$ and $\mathcal{P}_m$ is the set of all monic polynomials of degree $m$.

Let us suppose that $\mu$ is a measure with an infinite support on $\mathbb{R}$. If we let $P_{-1}(x;\mu)\equiv 0$ and $P_{0}(x;\mu)\equiv 1$ then there are two sequences $(a_n)_{n=1}^\infty$ and $(b_n)_{n=1}^\infty$ such that for $n\geq 0$ we have  $$ x P_n(x;\mu)= P_{n+1}(x;\mu)+b_{n+1} P_n(x;\mu)+a_n^2 P_{n-1}(x;\mu)$$ where $a_{n+1}>0$ and $b_{n+1}\in\mathbb{R}$. The coefficients  $(a_n, b_n)_{n=1}^\infty$ are called the recurrence coefficients associated with $\mu$. Both $(a_n)_{n=1}^\infty$ and $(b_n)_{n=1}^\infty$ are bounded sequences. Conversely, if we are given $(a_n, b_n)_{n=1}^\infty$ where $(a_n)_{n=1}^\infty$ and $(b_n)_{n=1}^\infty$ are bounded sequences with $a_n>0$ and $b_n\in\mathbb{R}$, then as a result of the spectral theorem there is a unique measure $\mu$ such that the associated recurrence coefficients are $(a_n, b_n)_{n=1}^\infty$. For a deeper discussion of the theory of orthogonal polynomials we refer the reader to \cite{Sim3,vanassche}.

The plan of the paper is as follows. In Section 2, we briefly summarize recent results from \cite{alp1,alp2} and well-known facts on the orthogonal polynomials associated with discrete measures. In Section 3, we discuss spacing of the zeros of orthogonal polynomials for fairly general measures. In the last section, we focus on the zero spacing of orthogonal polynomials for the equilibrium measure of the Cantor set $K(\gamma)$ which was introduced in \cite{gonc}.

\section{Preliminaries}
For the basic concepts of potential theory, we refer the reader to \cite{Ransford}. Convergence of measures is considered in the weak star topology. For a compact set $K\subset\mathbb{C}$ with the logarithmic capacity $\mathrm{Cap}(K)>0$, we denote the equilibrium measure of $K$ by $\rho_K$. 

Let $(f_n)_{n=1}^\infty$ be a sequence of nonlinear polynomials. Throughout, for each $n\in\mathbb{N}$ we use the following notation: $f_n(z)=\sum_{j=0}^{d_n}a_{n,j}\cdot z^j$ where $d_n\geq 2$, $a_{n,j}\in\mathbb{C}$ for $j=0,\dots,d_n$ and $a_{n,d_n}\neq 0$. The composition $f_n\circ f_{n-1}\dots \circ f_1$ will be denoted by $F_n$ and $\tau_n$ is used to denote the leading coefficient of $F_n$. The normalized counting measure on the roots (counting multiplicity) of $F_n(z)-a=0$ is denoted by $\nu_n^a$ where $a\in\mathbb{C}$. 

The next result is a more general version of Theorem 3.3 in \cite{alp2} and implies Theorem 2.8 of \cite{alp1}. In these theorems the limit sequence is the equilibrium measure of some prescribed set. 

\begin{theorem}
Let $(f_n)_{n=1}^\infty$ be a sequence of nonlinear polynomials. Suppose further that there is an $a\in\mathbb{C}$ such that $\nu_n^a\rightarrow \mu$ as $n\rightarrow\infty$ where $\mu$ is a probability measure and $\mathrm{supp}(\mu)$ is an infinite compact set in $\mathbb{C}$. Then we have the following identities:
	\begin{enumerate}[label={(\alph*})]
		\item $\displaystyle P_1(z;\mu)=z+ \frac{1}{d_1}\frac{a_{1, d_1-1}}{ a_{1,d_1}}.$
		\item $\displaystyle P_{d_1\dots d_l}(z;\mu)=\frac{1}{\tau_l}\left(F_l(z)+\frac{1}{d_{l+1}}\frac{a_{l+1,d_{l+1}-1}}{a_{l+1,d_{l+1}}}\right)$ for all $l\in\mathbb{N}$.
	\end{enumerate}
\end{theorem}
\begin{proof} The proof is almost same with the proof of Theorem 3.3 in \cite{alp2}. In the proof of Theorem 3.3 omit the first line, replace the equilibrium measure of $J_{(f_n)}$ by $\mu$ where it is necessary. Then we have the proof of this theorem.
\end{proof}

In the last section, we focus on a concrete family of measures but the techniques used in the last two sections are applicable to some extent for many other measures supported on $\mathbb{R}$ provided that the associated orthogonal polynomials satisfy $(a)$ and $(b)$ of the above theorem. For a systematic way to construct such measures we refer the reader to Section $4$ in \cite{alp2}.

If $\mu$ is a measure with an infinite support on $\mathbb{R}$ then the zeros of $P_n(\cdot;\mu)$ are simple and real. We enumerate the zeros $(x_{j,n}(\mu))_{j=1}^n$ of $P_n(\cdot;\mu)$ so that they satisfy 
\begin{equation*}
x_{1,n}(\mu)<x_{2,n}(\mu)<\dots<x_{n,n}(\mu).
\end{equation*}
Define $x_{0,n}(\mu)$ as the leftmost point and $x_{n+1,n}(\mu)$ as the rightmost point of $\mathrm{supp}(\mu)$, respectively. Then (see e.g. (2) in p. 358 of \cite{denis} and Theorem 2.3 in \cite{vanassche}), for $1\leq i \leq n$, $x_{i,n}(\mu)\in (x_{0,n}, x_{n+1,n})$. The next theorem (see for example Proposition 1.10 in \cite{baik} for a proof of it) will be used many times in the subsequent sections.

\begin{theorem}\label{ddd}
Let $\lambda$ be a measure with  $\mathrm{supp}(\lambda)=\{c_{i,r}\}_{i}\subset\mathbb{R}$ where $r\in\mathbb{N}$ with $r>1$ and $i=1,\dots, r$ provided that $c_{1,r}<c_{2,r}<\dots<c_{r,r}$. Then, the zeros of $P_s{(\cdot;\lambda)}$ lie in $(c_{1,r},c_{r,r})$ and they are real and simple where $1\leq s<r$. Moreover, in each interval $[c_{j,r},c_{j+1,r}]$ there is at most one zero of $P_s{(\cdot;\lambda)}$ where $j\in\{1,\dots,r-1\}$. 
\end{theorem}

We can reduce the infinite support case to the finite case by a classical technique. By doing that, we can use results such as Theorem \ref{ddd} which are valid for discrete measures. Assume that $\mu$ is a measure with an infinite support on $\mathbb{R}$ and let $r\in\mathbb{N}$. Then there is a unique measure $\mu^{(r)}$ with $\mathrm{supp}(\mu^{(r)})=\{x:\,P_r(x;\mu)=0\}$ such that for any polynomial $\pi$ with $\deg \pi\leq 2r-1$ we have
\begin{equation*}\label{ere}
\int \pi(x)\,d\mu(x)=\int\pi(x)\, d\mu^{(r)}(x).
\end{equation*} In particular,
\begin{equation}\label{gauss}
P_{s}(\cdot;\mu)= P_s(\cdot;\mu^{(r)})
\end{equation}holds for all  $1\leq s <r$ provided that $r>1$. See Theorem 2.5 in \cite{vanassche} for the proof. 

\section{Some general results}
For a measure $\mu$ having an infinite support on $\mathbb{R}$, let $Z_n(\mu):=\{x:\,P_n(x;\mu)=0\}$ and $Y_n(\mu):=\{x:\,P_n^\prime(x;\mu)=0\}$. For $n> 1$ with $n\in\mathbb{N}$, we define $M_n(\mu)$ by
\begin{equation*}
M_n(\mu):= \inf_{\substack{x,x^\prime\in Z_n(\mu)\\ x\neq x\prime}}|x-x^\prime|.
\end{equation*}
If $\mathrm{supp}(\mu)$ is a Cantor set on $\mathbb{R}$ then the maximal distance between the consecutive zeros of any associated orthogonal polynomial is not so interesting since this value is bounded below (see e.g. (iii) in p. 358 of \cite{denis}) by the half of the length of the largest gap of $\mathrm{supp}(\mu)$. We only discuss $M_n(\cdot)$ here. By $d(A,B)$ we mean the Euclidean distance between the sets $A,B\subset\mathbb{C}$.

\begin{proposition}\label{teo1}
Let $\mu$ be a measure with an infinite support on $\mathbb{R}$. Then for any fixed $l,m,n\in\mathbb{N}$ with $l>m>n>1$, we have $$d(Z_l(\mu),Z_m(\mu))= \inf_{\substack{1\leq i\leq l\\ 1\leq j \leq m}}|x_{i,l}(\mu)-x_{j,m}(\mu)|\leq M_n(\mu).$$
\end{proposition}
\begin{proof}
Let $\mu^{(l)}$ be defined as in Section 2. Then using \eqref{gauss} and Theorem \ref{ddd} for $\lambda=\mu^{(l)}$, $r=l$ and $s=m$, we have
$x_{1,l}(\mu)<x_{1,m}(\mu)<x_{m,m}(\mu)<x_{l,l}(\mu)$. By using Theorem \ref{ddd} for $\lambda=\mu^{(m)}$, $r=m$ and s=n in a similar manner, we see that
\begin{equation}\label{ara}
x_{1,l}(\mu)<x_{1,m}(\mu)<x_{i,n}<x_{m,m}(\mu)<x_{l,l}(\mu)
\end{equation} 
holds for all $i\in\{1,\dots,n\}$.

Assume to the contrary that for some $i\in\{1,\dots,n-1\}$, $[x_{i,n}(\mu),x_{i+1,n}(\mu)]$ does not contain an element from one of the sets $Z_l(\mu)$ and $Z_m(\mu)$. Without loss of generality, suppose that $Z_m(\mu)\cap [x_{i,n}(\mu),x_{i+1,n}(\mu)]=\emptyset$. Hence, by \eqref{ara}, $[x_{i,n}(\mu),x_{i+1,n}(\mu)]$ is contained in $(x_{j,m}(\mu), x_{j+1,m}(\mu))$
for some $j\in\{1,\dots,m-1\}$. This can not be true since by Theorem \ref{ddd}, $[x_{j,m}(\mu), x_{j+1,m}(\mu)]$ may contain at most one zero of $P_n{(\cdot;\mu^{(m)})}=P_n{(\cdot;\mu)}$. This gives the desired result.
\end{proof}

\begin{theorem}\label{tm}Let $\mu$ be a measure with an infinite support on $\mathbb{R}$ and let $(f_n)_{n=1}^\infty$ be a sequence of nonlinear polynomials. Assume further that there exists an $s_0\in\mathbb{N}$ such that  $f_l^\prime(0)=0$ for all $l>s_0$ and $P_{d_1\dots d_m}(\cdot;\mu)=F_m(\cdot)/\tau_m$ holds for all $m\in\mathbb{N}$. Then for all $k,k^\prime\in\mathbb{N}_0$ with $k>k^\prime$ the following holds: $$d(Z_{d_1\dots d_{s_0+k}}(\mu),Y_{d_1\dots d_{s_0+k}}(\mu))\leq d(Z_{d_1\dots d_{s_0+k}}(\mu),Z_{d_1\dots d_{s_0+k^\prime}}(\mu)).$$
\end{theorem}

\begin{proof}
Let us fix $k$ and $k^\prime$. If $k>k^\prime+1$ then for any $z\in\mathbb{C}$,
\begin{align*}
P^\prime_{d_1\dots d_{s_0+k}}(z;\mu)&=F^\prime_{{s_0+k}}(z)/\tau_{{s_0+k}}\\
&=((f_{s_0+k}\circ f_{s_0+k-1}\circ\dots\circ f_{s_0+k^\prime+1})\circ F_{{s_0+k^\prime}})^\prime(z)/\tau_{{s_0+k}}\\
&=\frac{F_{{s_0+k^\prime}}^\prime(z)\cdot(f_{s_0+k}\circ f_{s_0+k-1}\circ\dots\circ f_{s_0+k^\prime+1})^\prime( F_{{s_0+k^\prime}}(z))}{\tau_{s_0+k}}
\end{align*}
$$=\frac{(F_{{s_0+k^\prime}}^\prime\cdot(f_{s_0+k^\prime+1}^\prime\circ F_{{s_0+k^\prime}})\cdot(f_{s_0+k}\circ\dots\circ f_{s_0+k^\prime+2})^\prime(f_{s_0+k^\prime+1}\circ F_{{s_0+k^\prime}}))(z)}{\tau_{{s_0+k}}} $$

holds. Since $f_{s_0+k^\prime+1}^\prime(0)=0$, this implies that
\begin{equation*}\label{www}
Z_{d_1\dots d_{s_0+k^\prime}}(\mu)\subset Y_{d_1\dots d_{s_0+k}}(\mu),
\end{equation*}
 as $P_{d_1\dots d_{s_0+k^\prime}}(z;\mu)=F_{{s_0+k^\prime}}(z)/\tau_{{s_0+k^\prime}}.$ 
 If $k=k^\prime+1$ then
 \begin{align*}
 P^\prime_{d_1\dots d_{s_0+k}}(z;\mu)=\frac{F_{{s_0+k^\prime}}^\prime(z)\cdot(f_{s_0+k^\prime+1}^\prime\circ F_{{s_0+k^\prime}})(z)}{\tau_{{s_0+k}}},
 \end{align*}
 and $Z_{d_1\dots d_{s_0+k^\prime}}(\mu)\subset Y_{d_1\dots d_{s_0+k}}(\mu)$ similarly.
Thus, we have
$$d(Z_{d_1\dots d_{s_0+k}}(\mu),Y_{d_1\dots d_{s_0+k}}(\mu))\leq d(Z_{d_1\dots d_{s_0+k}}(\mu),Z_{d_1\dots d_{s_0+k^\prime}}(\mu)).$$
\end{proof}

The next proposition gives an upper bound for $M_n(\mu)$.
\begin{proposition}\label{uuu}
Let $\mu$ be a measure with an infinite support on $\mathbb{R}$ and let $n\in\mathbb{N}$ be given. Then for any $r\in\mathbb{N}$ satisfying $r>1$ and $r\geq n$, we have 
\begin{equation}\label{tre}
M_r(\mu)\leq\inf_{0\leq i\leq n-1}|x_{i+2,n}(\mu)-x_{i,n}(\mu)|.
\end{equation}
\end{proposition}
\begin{proof}
For $r=n$, \eqref{tre} follows from the definition of $M_r(\mu)$. So, let us pick an $r\in \mathbb{N}$ with $r>n$. 

Let $j\in\{0,\dots, n-1\}$ be chosen so that $$|x_{j+2,n}(\mu)-x_{j,n}(\mu)|=\inf_{0\leq i\leq n-1}|x_{i+2,n}(\mu)-x_{i,n}(\mu)|.$$ There are two cases to consider.

First, assume that $x_{j,n}(\mu)=x_{0,n}(\mu)$ or $x_{j+2,n}(\mu)=x_{n+1,n}(\mu)$ holds. Let $x_{j,n}(\mu)=x_{0,n}(\mu)$. Using Theorem \ref{ddd} for $\lambda=\mu^{(r)}$, we have $$x_{0,n}(\mu)<x_{1,r}(\mu)<x_{1,n}(\mu)<x_{r,r}(\mu)<x_{n+1,n}(\mu).$$ If we use Theorem \ref{ddd}, for $\lambda=\mu^{(r)}$ we see that $[x_{1,r}(\mu),x_{2,r}(\mu)]$ may contain at most one element from $\{x_{1,n}(\mu),\dots, x_{n,n}(\mu)\}$. Therefore, $M_r(\mu)\leq |x_{2,r}(\mu)-x_{1,r}(\mu)|\leq |x_{2,n}(\mu)-x_{0,n}(\mu)|$. For the case, $j+2=n+1$, a similar argumentation shows that $M_r(\mu)\leq |x_{r,r}(\mu)-x_{r-1,r}(\mu)|\leq |x_{n+1,n}(\mu)-x_{n-1,n}(\mu)|$.

Now, let us assume that $x_{j,n}(\mu)\neq x_{0,n}(\mu)$ and $x_{j+2,n}\neq x_{n+1,n}$. Using Theorem \ref{ddd} for $\lambda=\mu^{(r)}$, we have $$x_{1,r}(\mu)<x_{j,n}(\mu)<x_{j+1,n}(\mu)<x_{j+2,n}(\mu)<x_{r,r}(\mu).$$ Thus, there is a $k_1\in\mathbb{N}$ with $1<k_1<r$ such that $x_{k_1,r}(\mu)\in [x_{j,n}(\mu), x_{j+1,n}(\mu)]$ because otherwise there is an $i\in\{1,\dots, r-1\}$ such that $[x_{j,n}(\mu), x_{j+1,n}(\mu)]\subset (x_{i,r}(\mu),x_{i+1,r}(\mu) )$ and this would imply that $[x_{i,r}(\mu),x_{i+1,r}(\mu)]$ contains two zeros of $P_n(\cdot; \mu^{(r)})$ which is impossible again by Theorem \ref{ddd}. Assume to the contrary that $x_{k_1,r}(\mu)$ is the only zero of $P_r(\cdot; \mu)$ in $[x_{j,n}(\mu), x_{j+2,n}(\mu)]$. This implies that $(x_{k_1-1,r}(\mu), x_{k_1+1,r}(\mu))$ contains at least three zeros of $P_n(\cdot; \mu)$ but this is impossible by Theorem \ref{ddd} as $[x_{k_1-1,r}(\mu), x_{k_1+1,r}(\mu)]=[x_{k_1-1,r}(\mu^{(r)}), x_{k_1+1,r}(\mu^{(r)})]$ may contain at most $2$ zeros of $P_n(\cdot; \mu^{(r)})$ if we let $\lambda=\mu^{(r)}$. Hence there is a $k_2\in\mathbb{N}$ with $1<k_2<r$ and $k_2\neq k_1$ such that $x_{k_2,r}(\mu)\in [x_{j,n}(\mu), x_{j+2,n}(\mu)]$. Thus, $M_r(\mu)\leq |x_{k_2,r}(\mu)-x_{k_1,r}(\mu)|\leq |x_{j+2,n}(\mu)-x_{j,n}(\mu)|$. This shows that \eqref{tre} holds.
\end{proof}

\section{Zero spacing of orthogonal polynomials for a special family}

In this section, we study the spacing of the zeros of orthogonal polynomials for $\rho_{K(\gamma)}$ where $K(\gamma)$ is a Cantor set introduced in \cite{gonc}.

The construction and results in this and the next paragraph can be found in \cite{gonc}. Let, here and in the sequel, $\gamma_0:=1$ and $\gamma=(\gamma_k)_{k=1}^\infty$ be a sequence satisfying $0<\gamma_k<1/4$ for all $k\in\mathbb{N}$ provided that $\sum_{k=1}^\infty 2^{-k}\log{(1/{\gamma_k)}}<\infty$. We define $(f_n)_{n=1}^\infty$ by $f_1(z):=2z(z-1)/\gamma_1+1$ and $f_n(z):=z^2/(2\gamma_n)+1-1/(2\gamma_n)$ for $n>1$. Let $E_0:=[0,1]$ and $E_n:=F_n^{-1}([-1,1])$ where $F_n$ stands for $f_n\circ\dots\circ f_1$ as in Section 2. Then, $E_n$ is a union of $2^n$ disjoint non-degenerate compact intervals in $[0,1]$ and $E_n\subset E_{n-1}$ for all $n\in\mathbb{N}$. It turns out that, $K(\gamma):=\cap_{s=0}^\infty E_s$ is a non-polar Cantor set in $[0,1]$ where $\{0,1\}\subset K(\gamma)$. 

Let us look more carefully at the construction. We denote the connected components of $E_n$ by $I_{j,n}$ and the length of $I_{j_n}$ by $l_{j,n}$ for $j=1,\dots 2^n$, call these intervals as basic intervals of $n$-th level, define $a_{j,n}$ and $b_{j,n}$ by $[a_{j,n}, b_{j,n}]:= I_{j,n}$. Let $I_{1,0}:=E_0$ and $a_{j_1,n}>a_{j_2,n}$ if $j_1>j_2$. Then we have $I_{2j-1,n+1}\cup I_{2j,n+1}\subset I_{j,n}$ for all $n\in\mathbb{N}_0$ where $a_{2j-1, n+1}=a_{j,n}$ and $b_{2j,n+1}=b_{j,n}$. Denoting the gap $(b_{2j-1,n+1}, a_{2j,n+1})$ by $H_{j,n}$, for $1\leq j\leq 2^n$ and $n\in\mathbb{N}_0$, it follows that $$K(\gamma)= [0,1]\setminus (\cup_{n=0}^\infty\cup_{1\leq j\leq 2^n} H_{j,n}).$$

Using Theorem 11 in \cite{ger}, we see that $\rho_{E_n}(I_{j,n})=1/2^n$ for all $1\leq j \leq 2^n$ and $n\in\mathbb{N}_0$. Furthermore,  $\rho_{E_k}(I_{j,n})=1/2^n$ for $k>n$ since $I_{j,n}\cap E_k$ consists of $2^{k-n}$ basic disjoint intervals of $k$-th level. Since $(E_k)_{k=0}^\infty$ is a decreasing sequence of sets with $\cap_{s=0}^\infty E_s=K(\gamma)$, by part (ii) of Theorem A.16 in \cite{Sim4} (see also the proof of Corollary 3.2 in \cite{ag}) , it follows that $\rho_{K(\gamma)}(I_{j,n})=\rho_{K(\gamma)}(I_{j,n}\cap K(\gamma))=1/2^n$. The last in particular implies that $\rho_{K(\gamma)}([0,r])\in\mathbb{Q}$ for all $r\in\mathbb{R}$ with $r\notin K(\gamma)$.

It follows from the definition of equilibrium measure that $\mathrm{supp}{(\rho_{K(\gamma)})}\subset K(\gamma)$. We also have $K(\gamma)\subset \mathrm{supp}{(\rho_{K(\gamma)})}$ since for any $x\in K(\gamma)$ and $\epsilon>0$ the open ball $B_{\epsilon}(x)$ contains a basic interval $I_{j,n}$. From the above paragraph $\rho_{K(\gamma)}(I_{j,n}\cap K(\gamma))>0$ and therefore $K(\gamma)= \mathrm{supp}{(\rho_{K(\gamma)})}$.

In \cite{alp1}, it was shown that $P_{d_1\dots d_m}(z;\rho_{K(\gamma)})=P_{2^m}(z;\rho_{K(\gamma)})=F_m(z)/\tau_m$ for all $m\in\mathbb{N}$. Moreover, the recurrence coefficients have a simple form by Theorem 4.3 in \cite{alp1}. Let us denote the Lebesgue measure on the real line by $|\cdot|$. By lemma 6 in \cite{gonc}, $|\mathrm{supp}{(\rho_{K(\gamma)})}|=|K(\gamma)|=0$ if $0<\gamma_k<1/32$ for all $k\in\mathbb{N}$. Using spectral theory techniques developed for orthogonal polynomials this result was generalized in \cite{alp1}. If $\gamma$ satisfies $0< \gamma_k\leq 1/6$ for all $k$ then, by \cite{alp1}, $\liminf a_n=0$ where $(a_n)_{n=1}^\infty$ is the sequence of recurrence coefficients for $\rho_{K(\gamma)}$. The last, by \cite{dombr}, implies that $\rho_{K(\gamma)}$ has no non-trivial absolutely continuous part. Using the fact that $K(\gamma)= \mathrm{supp}{(\rho_{K(\gamma)})}$, we see that $\rho_{K(\gamma)}$ is purely singular continuous provided that $0<\gamma_k\leq 1/6$ for all $k\in\mathbb{N}$. Moreover, since $|\mathrm{supp}{(\rho_{K(\gamma)})}|>0$ guarantees that (see \cite{nazarov} and Section 1 of \cite{polto}) $\rho_{K(\gamma)}$ has a non-trivial absolutely continuous part, $|K(\gamma)|=0$ holds true for such a $\gamma$. 

Conversely, if $\gamma=(\gamma_k)_{k=1}^\infty$ satisfies $\sum_{k=1}^\infty \sqrt{(1-4\gamma_k)}<\infty$ then $|K(\gamma)|>0$ by \cite{alp2}. This was actually shown for the stretched version $K_1(\gamma)$ of $K(\gamma)$ but the same condition is valid for $K(\gamma)$. In the proof and the statement of Theorem 6.2 in \cite{alp2}, if we take $Z_0=1/2$, $\varepsilon_k=1/4-\gamma_k$ and 
put $K(\gamma)$ instead of $K_1(\gamma)$ then we have the condition that makes $K(\gamma)$ a Parreau-Widom set. At the end of the paper we return this condition and discuss it in a slightly more detailed way.

For a given sequence $\gamma=(\gamma_k)_{k=1}^\infty$, $f_1$ has two inverse branches $v_{1,1}, v_{2,1}:[-1,1]\rightarrow [0,1]$ with $v_{1,1}(t)=1/2-(1/2)\sqrt{1-2\gamma_1+2\gamma_1 t}$ and $v_2(t)= 1-v_1(t)$ where $(f_n)_{n=1}^\infty$ is defined as in the beginning of this section. For each $n>1$, $f_n$ has two inverse branches $v_{1,n}, v_{2,n}:[-1,1]\rightarrow [-1,1]$ such that $v_{1,n}(t)=\sqrt{1-2\gamma_n+2\gamma_n t}$ and $v_{2,n}(t)=-v_{1,n}(t)$. Note that $v_{1,n}([-1,1])\cap v_{2,n}([-1,1])=\emptyset$ for all $n\in\mathbb{N}$. By the fundamental theorem of algebra, for each $a\in\mathbb{C}$, $F_n(z)=a$ has at most $2^n$ different  solutions and therefore $\{v_{i_1,1}\circ\dots \circ v_{i_n,n}\}_{i_n\in\{1,2\}}$ gives the total set of inverse branches of $F_n=f_n\circ\dots \circ f_1$ on $[-1,1]$. In addition to this, for each $I_{j,n}$ there is a unique choice of $i_l\in\{1,2\}$ for $l=1,\dots,n$ giving $(v_{i_1,1}\circ \dots \circ v_{i_n,n})([-1,1])=I_{j,n}$ and in particular $I_{1,n}= (v_{1,1}\circ v_{1,2}\dots \circ v_{1,n})([-1,1])$. 

Now, let $u(t)=1/2-(1/2)\sqrt{1-4t}$ for $0\leq t\leq 1/4$. Then $(v_{i_1,1}\circ \dots \circ v_{i_n,n})(t)= g_1(\gamma_1\cdot g_2(\gamma_2 \dots \gamma_{n-1}\cdot g_n(\tilde{t})))$ for all $t\in [-1,1]$ where $g_l=u$ if $i_l=1$ and $g_l=1-u$ if $i_l=2$ for $l=1,\dots,n$, and $\tilde{t}= (\gamma_n-\gamma_n t)/2$. This last representation of inverse branches, which was used also in Section 3 of \cite{gonc}, simplifies the calculations since we have only two functions $u$ and $1-u$ instead of $2^n$ different functions. The function $u$ has a couple of nice properties that we will exploit many times. The last two of them are from \cite{gonc}.
\begin{proposition}\label{ccc}
The following hold:
\begin{enumerate}[label=(\alph*)]
\item $u$ and $u^\prime$ are strictly increasing. In particular, $u$ is strictly convex.	

\item $U_n:=u(\gamma_1\cdot u(\gamma_2 \dots \gamma_{n-1}\cdot u(\gamma_n)))= (1-\cos{(\pi/2^n)})/2$ for all $n\in\mathbb{N}$ where we take $\gamma_k= 1/4$ for all $k$. The number $(1-\cos{(\pi/2^n)})/2$ is the leftmost critical point of $F_n(z)$ and $F_n(z)/\tau_n=2^{-2^{n}} T_{2^n}(2z-1)$ by Example 1 of \cite{gonc} where $T_{2^n}$ is the $2^n$-th monic Chebyshev polynomial of the first kind.

\item $u(at)\leq a u(t)$ for all $0\leq t\leq 1/4$ and $0\leq a \leq 1$.

\item $u(t)\sqrt{1-4t}\leq t$ for $0\leq t\leq 1/4$.
\end{enumerate}
\end{proposition}

Next two lemmas easily follow from the properties mentioned in this section and the theorems from the previous one. 

\begin{lemma}\label{rere1} Let $\gamma=(\gamma_k)_{k=1}^\infty$ be given. For all $n\in\mathbb{N}$, we have

$$d(Z_{2^n}(\rho_{K(\gamma)}),Y_{2^n}(\rho_{K(\gamma)}))\geq$$ $$\inf_{\substack{i_j\in\{1,2\}\\ t\in\{-1,1\}}}|(v_{i_1,1}\circ \dots \circ v_{i_n,n})(t)-(v_{i_1,1}\circ \dots \circ v_{i_n,n})(0)|=$$

$$\inf_{\substack{g_i\in\{u,1-u\}\\ \tilde{t}\in\{0,\gamma_n\}}}|g_{1}(\gamma_1 \cdot g_{2}(\dots \gamma_{{n-1}}\cdot g_{n}(\tilde{t})))-g_{1}(\gamma_1 \cdot g_{2}(\dots \gamma_{{n-1}}\cdot g_{n}(\gamma_n/2)))|.$$
\end{lemma}
\begin{proof}
Let us choose an $n\in\mathbb{N}$. Note that, $|F_n(z)|>1$ for all $z$ satisfying $P^\prime_{2^n}(z;\rho_{K(\gamma)})= F_n^\prime(z)/\tau_n=0$. 
Moreover $F_n(I_{j,n})=[-1,1]$ and thus $Y_{2^n}(\rho_{K(\gamma)})\cap I_{j,n}=\emptyset$ for each $1\leq j \leq 2^n$. This implies that $d(x_{j,2^n}(\rho_{K(\gamma)}),Y_{2^n}(\rho_{K(\gamma)}))\geq d(x_{j,2^n}(\rho_{K(\gamma)}), \{a_{j,n},b_{j,n}\})$. Hence, the first inequality holds. The second one follows from the definition of $g_i$.
\end{proof}

\begin{lemma}\label{rere}
Let $\gamma=(\gamma_k)_{k=1}^\infty$ and $r\in\mathbb{N}$ be given. Then, for any $k\in\mathbb{N}$ with $r\geq 2^k$,
\begin{align*}
M_r(\rho_{K(\gamma)})&\leq \inf_{0\leq i\leq 2^k-1} |x_{i+2,2^k}(\rho_{K(\gamma)})-x_{i,2^k}(\rho_{K(\gamma)})|
\\ 
&\leq |x_{2,2^k}(\rho_{K(\gamma)})-x_{0,2^k}(\rho_{K(\gamma)})|\\ &\leq l_{1,k-1},
\end{align*}
holds.
\end{lemma}
\begin{proof}
By using Proposition \ref{uuu} for $\mu= \rho_{K(\gamma)}$ and $n=2^k$, it can be seen that the first inequality holds. The last one holds true since $[x_{0,2^k}(\rho_{K(\gamma)}), x_{2,2^k}(\rho_{K(\gamma)})]=[0,x_{2,2^k}(\rho_{K(\gamma)})]\subset I_{1,k-1}$ for all $k\in\mathbb{N}$.
\end{proof}

Now, let us prove an auxiliary result which is an analogue of Lemma 5.2 from \cite{alp2}. For a given $\gamma=(\gamma_k)_{k=1}^\infty$, we denote the product $\gamma_0\dots \gamma_n$ by $\delta_n$ for $n\in\mathbb{N}_0$.

\begin{lemma}\label{yty}
Let $\gamma=(\gamma_k)_{k=1}^\infty$ be given. Then 
\begin{equation}\label{bbb}
\delta_s\leq l_{1,s}\leq \frac{\pi^2}{4}\cdot \delta_s
\end{equation} holds for all $s\in\mathbb{N}_0$. 
\end{lemma}
\begin{proof}
For $s=0$, \eqref{bbb} holds trivially. So, let $s\geq 1$. Observe that
\begin{align}
l_{1,s}&=u(\gamma_1\cdot u(\gamma_2 \dots \gamma_{s-1}\cdot u(\gamma_s)))-u(\gamma_1\cdot u(\gamma_2 \dots \gamma_{s-1}\cdot u(0)))\\
&=u(\gamma_1\cdot u(\gamma_2 \dots \gamma_{s-1}\cdot u(\gamma_s)))\label{zzz}\\
&=u((4\gamma_1)\cdot (1/4)\cdot u((4\gamma_2)\cdot (1/4)\dots \gamma_{s-1}\cdot u((4\gamma_s)\cdot (1/4))))\\
&\leq 4^s \delta_s U_s.\label{aaaa}
\end{align} 
The three equalities above are straightforward and the last inequality follows from the parts $(b)$ and $(c)$ of Proposition \ref{ccc}. Since $1-\cos{x}\leq x^2/2$ for all $x\in[0,\infty)$, we have $$U_s=(1-\cos{(\pi/2^s)})/2\leq (4^{-s}\pi^2)/4.$$ Using this and \eqref{aaaa}, the right part of \eqref{bbb} follows. Using $(c)$ in Proposition \ref{ccc}, it is elementary to see that $\delta_s\leq u(\gamma_1\cdot u(\gamma_2 \dots \gamma_{s-1}\cdot u(\gamma_s)))$. This and \eqref{zzz} gives the left part of \eqref{bbb}.
\end{proof}
The next lemma will allow us to find a lower bound for $M_n(\rho(K(\gamma)))$.

\begin{lemma}\label{roro}
Let $\gamma=(\gamma_k)_{k=1}^\infty$ be given. Then for any choice of $g_i\in\{u,1-u\}$, $i=1,\dots, n$, we have $$\inf_{ \tilde{t}\in\{0,\gamma_n\}}|g_1(\gamma_1\cdot g_2( \dots \gamma_{n-1}\cdot g_n(\tilde{t})))-g_1(\gamma_1\cdot g_2( \dots \gamma_{n-1}\cdot g_n(\gamma_n/2)))|\geq$$ 
\begin{equation}\label{kkk}
\geq u(\gamma_1\cdot u(\gamma_2\cdot\dots \gamma_{n-1}\cdot u(\gamma_n/2)))\geq l_{1,n+1}\geq \delta_{n+1}.
\end{equation}

\end{lemma}

\begin{proof}
Let $n\in\mathbb{N}$ be given. Then $l_{1,n+1}\geq\delta_{n+1}$ by Lemma \ref{yty}. Since $u(t)\leq 1/2$, we have $\gamma_n\cdot u(\gamma_{n+1})\leq {\gamma_n}/2$. Using the part $(c)$ of Proposition \ref{ccc} we see that $$u(\gamma_1\cdot u(\gamma_2 \dots u(\gamma_{n}/2)))\geq u(\gamma_1\cdot u(\gamma_2 \dots \gamma_{n}\cdot u(\gamma_{n+1})))= l_{1,n+1}$$ holds and thus the second inequality in \eqref{kkk} follows. 

In order to prove the first inequality in \eqref{kkk}, it suffices to show that for a given $c\in[0,\gamma_n/2]$, and $g_i\in\{u,1-u\}$, $i=1,\dots,n$, the following inequality holds:
\begin{equation}\label{bubu}
|g_1(\gamma_1 \dots \gamma_{n-1}\cdot g_n(c+\gamma_n/2))-g_1(\gamma_1 \dots \gamma_{n-1}\cdot g_n(c))|\geq u(\gamma_1\cdot u(\gamma_2\cdot\dots \gamma_{n-1}\cdot u(\gamma_n/2))).
\end{equation}

Let $$q_{n+1-k}:=g_{k}(\gamma_{k}\cdot g_{k+1}(\dots \gamma_{n-1}\cdot g_n(c+\gamma_n/2))),$$ $$t_{n+1-k}:=g_{k}(\gamma_{k}\cdot g_{k+1}(\dots \gamma_{n-1}\cdot g_n(c))),$$  $$r_{n+1-k}:=u(\gamma_{k}\cdot u(\gamma_{k+1}\dots \gamma_{n-1}\cdot u(\gamma_n/2))),$$ $$s_k:=|q_k-t_k|,$$ for $k=1,\dots,n$. If $n=1$, $s_n\geq r_n$ since from the strict convexity of $u$ we have 
\begin{equation}\label{kekeke}
|g_1(c+\gamma_1/2)-g_1(c)| =u(c+\gamma_1/2)-u(c)\geq u(\gamma_1/2)-u(0).
\end{equation}
Suppose that $n>1$. We want to show that $s_n\geq r_n$. Let us proceed by induction. For $k=1$, $s_k\geq r_k$ by \eqref{kekeke}. Suppose that the induction hypothesis holds for all $k=1,\dots m$ provided that $m\leq n-1$.   Using, the strict convexity of $u$ in \eqref{4} and the fact that $u$ is increasing in \eqref{6} we have 
\begin{align}
\label{1} s_{m+1}&=|q_{m+1}-t_{m+1}|\\
\label{2}&=|g_{n-m}(\gamma_{n-m}\cdot q_m)-g_{n-m}(\gamma_{n-m}\cdot t_m)|\\
\label{3}&=|u(\gamma_{n-m}\cdot q_m)-u(\gamma_{n-m}\cdot t_m)|\\
\label{4}&\geq  u(\gamma_{n-m}\cdot |q_m-t_m|)\\
\label{5}&=u(\gamma_{n-m}\cdot s_m)\\
\label{6}&\geq u(\gamma_{n-m}\cdot r_m)\\ 
\label{7}&=r_{m+1}.
\end{align}
Hence, $s_n\geq r_n$ holds if we take $m=n-1$ above. This gives \eqref{bubu} and completes our proof.
\end{proof}

Eventually, we are ready to prove the main result of the paper. 

\begin{theorem}\label{last}
Let $\gamma=(\gamma_k)_{k=1}^\infty$ and $n\in\mathbb{N}$ with $n>1$ be given. Furthermore, let $s$ be the integer satisfying $2^{s-1}\leq n < 2^s$. Then 
\begin{equation}\label{eq1}
\delta_{s+2}\leq M_n(\rho_{K(\gamma)})\leq \frac{\pi^2}{4}\cdot \delta_{s-2}
\end{equation}
holds. In particular, if $\displaystyle\inf_{k} \gamma_k=c>0$ then we have
\begin{equation}\label{eq2}
c^2\cdot \delta_{s}\leq M_n(\rho_{K(\gamma)})\leq \frac{\pi^2}{4c^2}\cdot \delta_{s}.
\end{equation}

\end{theorem}
\begin{proof}
First, let us prove \eqref{eq1}. Recall that for all $m>1$, $f^\prime_m(0)=0$ and $d_i=2$ for all $i\in\mathbb{N}$. Using Theorem \ref{tm} for $\mu=\rho_{K(\gamma)}$, $s_0=s$, $k=1$ and $k^\prime=0$, we have 
\begin{equation*}\label{trt}
 d(Z_{2^{s+1}}(\rho_{K(\gamma)}),Y_{2^{s+1}}(\rho_{K(\gamma)}))\leq d(Z_{2^{s+1}}(\rho_{K(\gamma)}),Z_{2^{s}}(\rho_{K(\gamma)})).
\end{equation*}
By using Proposition \ref{teo1}, for $\mu=\rho_{K(\gamma)}$, $l=2^{s+1}$, $m=2^{s}$, it can be seen that
\begin{equation*}\label{trt2}
d(Z_{2^{s+1}}(\rho_{K(\gamma)}),Z_{2^s}(\rho_{K(\gamma)}))\leq M_n(\rho_{K(\gamma)}).
\end{equation*}
Hence $d(Z_{2^{s+1}}(\rho_{K(\gamma)}),Y_{2^{s+1}}(\rho_{K(\gamma)}))\leq M_n(\rho_{K(\gamma)})$ holds. By Lemma \ref{rere1} and Lemma \ref{roro}, the term on the left part of this last inequality is bounded below by $\delta_{s+2}$. This gives the first inequality in \eqref{eq1}. 

Using Lemma \ref{rere} for $r=n$ and $k=s-1$ and then Lemma \ref{yty}, we deduce that $$M_n(\rho_{K(\gamma)})\leq l_{1,s-2}\leq(\pi^2/4) \delta_{s-2}.$$ This completes the proof of \eqref{eq1}.

Combining the first inequality of \eqref{eq1} and the fact that $c^2\leq\gamma_{s+1}\cdot \gamma_{s+2}$, the first inequality in \eqref{eq2} follows. Observe that $1/c^2\geq 1/(\gamma_{s-1}\cdot\gamma_{s-2})$. Hence, the second inequality in \eqref{eq1} implies that of \eqref{eq2}. So, we are done.
\end{proof}

\begin{remark}
If there is a $c$ with $0< c\leq\gamma_k\leq 1/6$ for all $k\in\mathbb{N}$, then $\rho_{K(\gamma)}$ is purely singular continuous. Moreover \eqref{eq2} is satisfied and $|K(\gamma)|=0$ holds. 

If $\gamma=(\gamma_k)_{k=1}^\infty$ satisfies $\sum_{k=1}^\infty \sqrt{(1-4\gamma_k)}<\infty$ then there is a $c$ such that $0<c \leq\gamma_k<1/4$ holds for all $k\in\mathbb{N}$ and $K(\gamma)$ is a Parreau-Widom set (see e.g. \cite{christiansen} for the definition) by \cite{alp2}. Thus, (see \cite{christiansen,polto} for the proof) $\rho_{K(\gamma)}$ and the Lebesgue measure restricted to $K(\gamma)$ are mutually absolutely continuous . Yet, \eqref{eq2} gives a pretty accurate description of $(M_n(\rho_{K(\gamma)}))_{n=2}^\infty$.

For a $\gamma=(\gamma_k)_{k=1}^\infty$ with $\sum_{k=1}^\infty \gamma_k< \infty$, $\rho_{K(\gamma)}$ and a special Hausdorff measure $\Lambda_h$ defined in \cite{ag} are mutually absolutely continuous. In this case, $K(\gamma)$ is of Hausdorff dimension zero (can be seen by using 2.3 in \cite{ag}) and we only have \eqref{eq1}.
\end{remark}
The term $\delta_n$ plays an important role to characterize the smoothness properties of the Green function $G_{\overline{\mathbb{C}}\setminus K(\gamma)}$ with pole at infinity (see Section 5 in \cite{alp2} and \cite{gonc2}). For an overview of the smoothness properties of the Green function for the complement of homogeneous Cantor sets, we refer the reader to \cite{tot}. It is unclear that whether there is, in general, a meaningful relation between the spacing properties of orthogonal polynomials for $\rho_K$ and the supremum of the H\"{o}lder exponents making $G_{\overline{\mathbb{C}}\setminus K}$ H\"{o}lder continuous  provided that $K\subset\mathbb{R}$ is a non-polar compact set.

It seems that similar results to Theorem \ref{last} can be obtained for the equilibrium measure of the Julia set $J(f)\subset \mathbb{R}$ of a quadratic polynomial of the form $f(z)=z^2-c$ with $c>2$. If we let $f_n\equiv f$ for all $n\in\mathbb{N}$ then by \cite{Barnsley1}, $F_{n}(z)=P_{2^n}(z;\rho_{J(f)})$ and $f_{n}^\prime(0)=0$ holds. Besides, the inverse branches of $f$ are $f_{\pm}(t)=\pm \sqrt{t+c}$ and possibly they lead to similar calculations.

\end{document}